%% file: root.tex
\documentclass[journal,twoside,web]{ieeecolor}
\usepackage[dvipsnames]{xcolor}
\usepackage{generic}
\usepackage{lcsys}

\usepackage{graphicx}
\usepackage{epsfig}
\usepackage{subfigure}
\usepackage{cite}
\usepackage{textcomp}

\usepackage{amsmath,amssymb,amsfonts}
\usepackage{mathtools}
\usepackage{mathrsfs}
\usepackage{stmaryrd}
\usepackage{bm}
\usepackage{physics}

\usepackage{algorithm}
\usepackage{algpseudocode}

\usepackage{float}
\usepackage{tabularx}
\usepackage{booktabs}
\usepackage{placeins}
\usepackage{pgfplots}
\usepackage{pgfplotstable}
\usepgfplotslibrary{statistics}
\usepgfplotslibrary{groupplots}
\usepgfplotslibrary{dateplot}
\usepgfplotslibrary{fillbetween}
\usetikzlibrary{calc}
\pgfplotsset{compat=1.18}

\usepackage{tcolorbox}
\usepackage{epigraph}
\usepackage{matlab-prettifier}
\usepackage{setspace}
\usepackage{musicography}
\usepackage{xurl}

\usepackage[colorlinks=true, linkcolor=black, citecolor=black, urlcolor=black]{hyperref}

\usepackage{caption}
\usepackage[acronym]{glossaries}
\makeglossaries
\glsdisablehyper

\newacronym{LTI}{LTI}{Linear Time-Invariant}
\newacronym{LTV}{LTV}{Linear Time-Varying}
\newacronym{DeePC}{DeePC}{Data-enabled Predictive Control}
\newacronym{SVD}{SVD}{singular value decomposition}
\newacronym{PCA}{PCA}{principal component analysis}
\newacronym{GREAT}{GREAT}{Grassmannian REcursive Algorithm for Tracking}
\newacronym{VTVL}{VTVL}{Vertical Takeoff, Vertical Landing}
\newacronym{4SID}{4SID}{subspace-based methods for state-space system identification}
\newacronym{RGD}{RGD}{Riemannian gradient descent}
\newacronym{SISO}{SISO}{single-input single-output}
\newacronym{ARX}{ARX}{AutoRegressive model with eXogenous inputs}
\newacronym{FROSL}{FROSL}{Flag Recursive Online Subspace Learning}

\input{user_command}

\begin{document}

\def\BibTeX{{\rm B\kern-.05em{\sc i\kern-.025em b}\kern-.08em
    T\kern-.1667em\lower.7ex\hbox{E}\kern-.125emX}}

\markboth{\journalname, VOL. XX, NO. XX, XXXX 2026}
{Jin and Coulson: On the Sensitivity of the Subspace Predictor to Behavioral Perturbations}

\title{On the Sensitivity of the Subspace Predictor to Behavioral Perturbations}

\author{
Dian Jin$^{1}$, Jeremy Coulson$^{1}$
\thanks{$^{1}$Both authors are with the Department of Electrical and Computer Engineering,
University of Wisconsin--Madison, Madison, WI, USA.
Emails: \{djin38, jeremy.coulson\}@wisc.edu.}
}

\maketitle
\thispagestyle{empty}
\pagestyle{empty}

\begin{abstract}
Behavioral systems define discrete-time \gls{LTI} systems in terms of a set of trajectories, which forms a linear subspace.
This subspace underlies the subspace predictor used in data-driven prediction and control.
In practice, such subspaces are typically represented through data matrices.
For robustness certification and uncertainty quantification, however, these matrix representations are coordinate-dependent and therefore do not provide a coordinate-free way to quantify uncertainty.
In this work, we derive an explicit prediction error bound in terms of behavioral distance between the true subspace and an estimate, showing that the predictor is locally Lipschitz with respect to behavioral perturbations.
We also present a one-step prediction error bound that is relevant for receding-horizon implementations, which becomes computable when combined with existing behavioral-distance certificates.
Numerical studies show that our bound is tighter than an existing data-matrix perturbation bound, and remains computable, though more conservative, when combined with an existing behavioral distance certificate.
\end{abstract}



\input{final_submission/final_version}


\bibliographystyle{IEEEtran}
\bibliography{ref}

\end{document}

%% file: user_command.tex
\newcommand{\up}{{u_{\mathrm{p}}}}
\newcommand{\yp}{{y_{\mathrm{p}}}}
\newcommand{\uf}{{u_{\mathrm{f}}}}
\newcommand{\yf}{{y_{\mathrm{f}}}}

\newcommand{\im}{\mathrm{im~}}

\newcommand{\ini}{\mathrm{ini}}

\newcommand{\rmd}{\mathrm{d}}

\newcommand{\Tf}{T_\mathrm{f}}
\newcommand{\Tini}{T_\mathrm{ini}}

\newcommand{\pred}{\mathrm{pred}}

\newtheorem{theorem}{Theorem}
\newtheorem{proposition}{Proposition}

\newtheorem{lemma}{Lemma}
\newtheorem{definition}{Definition}

\newtheorem{problem}{Problem}

\newtcolorbox{noframebox}{
  colback=PeachRed!20,   
  boxrule=0pt,       
  sharp corners,      
  left=4pt, right=4pt, top=6pt, bottom=6pt,
}

\newcommand{\ga}{\gamma}

\newcommand{\R}{\mathbb{R}}

\newcommand{\Z}{\mathbb{Z}}
\newcommand{\Zp}{\mathbb{Z}_{\geq 0}}

\newcommand{\rmO}{\mathrm{O}}

\newcommand{\bmU}{\mathbf{U}}

\newcommand{\bmV}{\mathbf{V}}

\newcommand{\calB}{\mathcal{B}}

\newcommand{\calH}{\mathcal{H}}

\newcommand{\calN}{\mathcal{N}}
\newcommand{\calO}{\mathcal{O}}

\newcommand{\calS}{\mathcal{S}}
\newcommand{\calT}{\mathcal{T}} 

\renewcommand{\hat}{\widehat}

\definecolor{WiscRed}{RGB}{180, 42, 26}

\definecolor{PeachRed}{RGB}{231, 143, 129}

\DeclareMathOperator*{\diag}{diag}  

\newcommand{\Gr}{\mathrm{Gr}}   


%% file: final_submission/final_version.tex
\section{Introduction}
\IEEEPARstart{P}{rediction} from data is a canonical problem across machine learning, signal processing, and control, where a common theme is to exploit low-complexity structure for forecasting and decision making. In systems and control, this viewpoint has led to subspace-based methods for system identification and data-driven prediction/control \cite{van2012subspace, markovsky2008data, coulson2019data}. Behavioral systems theory \cite{willems1986time, markovsky2021behavioral} provides a trajectory-based representation of finite-horizon \gls{LTI} systems: the set of admissible trajectories forms a linear subspace that can be constructed directly from data. The subspace predictor \cite{markovsky2008data} uses this behavioral subspace to predict future outputs from past input-output data and prescribed future inputs.

Robustness and uncertainty quantification for data-driven prediction can be approached either at the data-matrix level \cite{kaviani2025uncertainty} or at the behavioral-subspace level \cite{behavioraluq}. 
The latter viewpoint has become increasingly relevant in recent work, especially when the object produced by the algorithm is itself a subspace estimate. For example, online subspace estimation methods track time-varying behavioral subspaces \cite{sasfi2026great, jin2025online}, while robust least-squares formulations for data-driven control \cite{bharadwaj2025robust} model uncertainty through geometric neighborhoods of a nominal behavior. In these settings, uncertainty is most naturally described by the distance between the true and estimated behavioral subspaces \cite{fazzi2023distance}. These developments motivate a coordinate-free sensitivity analysis directly at the level of behaviors. Since subspace uncertainty is naturally measured by distances such as  gap metric \cite{ball2006equivalence, bian2008intrinsic} (also see \cite{padoan2025distances} and references therein)  or chordal distance, whereas prediction performance is measured by Euclidean output error, it remains necessary to quantify how geometric discrepancies between behaviors affect the resulting predictions.

Recent work \cite{kaviani2025uncertainty} derives 
computable prediction error bounds under data matrix perturbations caused by additive output noise.
In contrast, we study perturbations at the level of behavior itself. This geometric viewpoint is more general than an additive-noise model: it can capture uncertainty arising from truncated \gls{SVD} preprocessing \cite{alsalti2024robust, kaviani2025uncertainty}, online subspace identification \cite{sasfi2026great}, and geometric uncertainty sets used in robust data-driven control \cite{bharadwaj2025robust}. 
Our goal is to quantify how a geometric error between the true and estimated behavior propagates to the predicted output. 
When combined with computable chordal
distance bounds such as \cite[Theorem 9]{alsalti2024robust} or with subspace-tracking guarantees \cite{sasfi2026great}, our result yields
prediction-error certificates that are computable from data and noise bounds.

The contribution of this work is twofold.
First, we quantify in Theorem \ref{thm:error-bound} the sensitivity of the subspace predictor to behavioral perturbations by deriving an explicit upper bound on the prediction error in terms of the behavioral distance between the ground-truth subspace and its approximation. The bound also depends on a quantitative observability property of the system.
Second, we provide a data-dependent prediction-error bound whose constants are computable from the estimated subspace in Theorem \ref{thm:one-step-pred-error-known}.
Thus, our results convert geometric uncertainty certificates into prediction-error certificates.


The letter is organized as follows: Section \ref{sec:motivation} formalizes the problem of interest. Section \ref{sec:error-bound} contains the main results (Theorem \ref{thm:error-bound} and  Proposition \ref{thm:sp-rotinvar}). 
Section \ref{sec:numerical} empirically studies the sensitivity of the subspace prediction in terms of the behavioral distance through a numerical case study.

\textit{Notation:}
The set of all nonnegative integers is denoted $\Zp$. 
Given $i, j, T \in \Zp$ with $i<j$ and a sequence 
$\smash{\{z(t)\}_{t=0}^{T-1}} \subset \R^n$, define $[i,j]=[i, i+1, \dots, j-1, j]$ and $z_{[i,j]} \coloneqq [z(i)^\top ~ z(i+1)^\top ~ \cdots ~ z(j)^\top]^\top.$
We use $\| \cdot \|_2$ to denote both the Euclidean norm for vectors and the spectral norm for matrices, and $\| \cdot \|_F$ to denote the Frobenius norm. For $k \leq T$, let $\calH_k(z_{[0,T-1]})$ denote the block-Hankel matrix whose $(\ell+1)^{\text{th}}$ column is
$z_{[\ell,\ell+k-1]}$, for $\ell=0,\ldots,T-k$.
Given $T \in \Zp$, the sequence $z_{[0,T-1]}$ is called \textit{persistently exciting} of order $k$ if $\calH_k(z_{[0,T-1]})$ has full row rank. The Moore-Penrose inverse of a matrix $M$ is denoted by $M^\dag$.
We use boldface $\bmU$ to denote a linear subspace of $\R^q$, and plain $U$ to denote a spanning matrix. These notations are connected by $\im U = \bmU$, and will be used interchangeably.
The set of orthogonal matrices is denoted $\rmO(d)=\{U \in \R^{d \times d}: U^\top U = I\}$. The set of all singular values of a matrix $U$ is denoted by $\sigma(U)$.
The $i^\text{th}$ singular value of a matrix $U$ is denoted by $\sigma_i(U)$. The diagonal matrix with diagonal entries $\ga_1,\cdots,\ga_d$ is denoted $\diag(\ga_1,\cdots,\ga_d)$. 
The Grassmannian is defined as $\Gr(q, d)\coloneqq \{\bmU $ is a linear subspace of $\R^q: \dim \bmU = d \}.$ Let $\bmU, \bmV \in \Gr(q, d)$ and let $U, V \in \R^{q \times d}$ be orthonormal bases spanning $\bmU$ and $\bmV$, respectively. 
    Consider the compact \gls{SVD} $U^\top V = P \Sigma Q^\top$, where $\Sigma=\diag(\sigma_1,\dots,\sigma_d)$ with $\sigma_1 \geq \dots \geq \sigma_d \geq 0$.
    The principal angles $\theta_1, \dots, \theta_d \in [0, \pi/2]$ between $\bmU$ and $\bmV$ are given by \cite{golub2013matrix}
    $\cos \theta_i = \sigma_{i}$. The \textit{chordal distance} between $\bmU, \bmV$ is defined equivalently by 
    \begin{equation}\label{eq:chord}
        d(\bmU, \bmV) = \left(\sum_{i=1}^d \sin^2 \theta_i\right)^{\frac{1}{2}} = 
        \frac{1}{\sqrt{2}}\norm{UU^\top - VV^\top}_F.
    \end{equation}
Note that the chordal distance is  coordinate-free: $d(\im (UQ), \im (VQ)) = d(\im U, \im V)$ for all $Q \in \rmO(d)$.
The Gaussian distribution with mean $0$ and covariance $\Sigma \in \R^{d \times d}$ is denoted $\calN(0, \Sigma)$.
\section{Problem Statement}\label{sec:motivation}
Consider the observable \gls{LTI} dynamical system 
\begin{equation}\label{eq:LTI-1}
\begin{aligned}
    x(t+1) &= A x(t) + B u(t), \\
    y(t) &= C x(t) + D u(t),
\end{aligned}
\end{equation}
where $A \in\R^{n\times n}$, $B\in\R^{n\times m}$, $C\in\R^{p\times n}$, $D\in\R^{p\times m}$, $x(t)$ is the state, $u(t)$ is the input, $y(t)$ is the output at time $t \in \Zp$. The \textit{restricted behavior} of \eqref{eq:LTI-1} is defined as
\begin{equation*}
    \begin{aligned}
    &\calB_{[0,L-1]} = \bigg\{
    \left(u_{[0,L-1]}, y_{[0,L-1]} \right) \in \R^{(m+p)L} 
    \bigg|~ \text{there exists} \\
    &x_{[0,L]} \in \R^{n(L+1)} 
    \text{ such that \eqref{eq:LTI-1} holds for all $t \in [0,L-1]$} \bigg\}.
    \end{aligned}
\end{equation*}
Any $(u_{[0,L-1]}, y_{[0,L-1]}) \in \calB_{[0,L-1]}$ is called a \textit{trajectory} of system \eqref{eq:LTI-1}.
The following lemma \cite{markovsky2021behavioral} shows that $\calB_{[0,L-1]}$ is a linear subspace of $\R^{(m+p)L}$ and provides a data matrix representation.
\begin{lemma}[{\cite[Corollary 19]{markovsky2022identifiability}}]\label{thm:fund-lemma}
    Consider system \eqref{eq:LTI-1}. 
    Let $(A, B)$ be controllable and $L \geq n$. Then $\calB_{[0,L-1]}$ is a linear subspace of dimension $\dim \calB_{[0,L-1]} = mL+n$. Moreover, let $T \in \Z_{>0}$ and $(u_{[0,T-1]}, y_{[0,T-1]}) \in \calB_{[0,L-1]}$ with $u_{[0,T-1]}$ being persistently exciting of order $n+L$, then
    \[
    \operatorname{im}
    \begin{bmatrix}
        \calH_L(u_{[0,T-1]}) \\
        \calH_L(y_{[0,T-1]}) 
    \end{bmatrix} = \calB_{[0,L-1]}.
    \]
\end{lemma}
\noindent
Lemma \ref{thm:fund-lemma} provides the foundation for viewing systems over finite horizons as subspaces \cite{padoan2025distances}. 
The data-driven prediction problem posed in \cite{markovsky2008data} aims to find the output sequence of an unknown dynamical system corresponding to an input sequence 
based on a past input-output trajectory collected from the system, without explicitly identifying a state-space model. 
We now introduce the subspace predictor, which is the key ingredient of data-driven prediction and the main object studied in this work.
\begin{definition}\label{def:subspace-predictor}
    Let $X \in \R^{(m+p)(\Tini+\Tf) \times r}, r \in \Z_{>0}$. Let $(u_\ini, u, y_\ini) \in \R^{m\Tini + m\Tf + p\Tini}$. 
    Partition $X$ into a block matrix
    \[
    X = \begin{bmatrix}
    X_{\up}^\top &
    X_{\uf}^\top &
    X_{\yp}^\top &
    X_{\yf}^\top 
    \end{bmatrix}^\top,
    \]
    where $X_{\up} \in \R^{m\Tini \times r}, X_{\uf} \in \R^{m\Tf \times r}, X_{\yp} \in \R^{p\Tini \times r}, X_{\yf} \in \R^{p\Tf \times r}.$
    The subspace predictor is defined as the mapping 
    \begin{equation}\label{eq:subspace-predictor}
    \begin{aligned}
        &\calS: \R^{(m+p)(\Tini+\Tf) \times r} \times \R^{m\Tini + m\Tf + p\Tini} \to \R^{p\Tf}, \\
    &\calS(X, u_\ini, u, y_\ini) =  
    X_{\yf}
    \begin{bmatrix}
    X_{\up} \\[2pt]
    X_{\uf} \\
    X_{\yp}
    \end{bmatrix}^\dag 
    \begin{bmatrix}
        u_\ini \\ u \\ y_\ini
    \end{bmatrix}.        
    \end{aligned}
    \end{equation}
\end{definition}
\noindent

Throughout this paper, the subscripts $\up, \uf, \yp, \yf$ always refer to this block partition for any matrices with $(m+p)(\Tini+\Tf)$ rows.
When the vector $(u_\ini, u, y_\ini)$ is clear from the context, we equivalently write $y^\pred_{[0,\Tf-1]}(X) \coloneqq (y^\pred_0(X), \dots, y^\pred_{\Tf-1}(X))$ for $\calS(X, u_\ini, u, y_\ini)$. 
The subspace predictor can be used to predict future output trajectories of a dynamical system given future input $u \in \R^{m\Tf}$ and an initial trajectory $(u_\ini, y_\ini) \in \calB_{[0,\Tini-1]}$ \cite{markovsky2008data}. 
We wish to study how the perturbation to $\calB_{[0,\Tini+\Tf-1]}$ affects the subspace prediction. 
Let $\hat{\calB}_{[0,\Tini+\Tf-1]}$ be the behavior of an approximate system
\begin{equation}\label{eq:LTI-2}
\begin{aligned}
    \hat{x}(t+1) &= \hat{A} \hat{x}(t) + \hat{B} u(t), \\
    \hat{y}(t) &= \hat{C} \hat{x}(t) + \hat{D} u(t),
\end{aligned}
\end{equation}
where $\hat{A} \in\R^{n\times n}$, $\hat{B}\in\R^{n\times m}$, $\hat{C}\in\R^{p\times n}$, $\hat{D}\in\R^{p\times m}$.
Let $u^\rmd_{[0,T-1]}$ be persistently exciting of order $n+\Tini+\Tf$, 
let $(u^\rmd_{[0,T-1]}, y^\rmd_{[0,T-1]}) \in \calB_{[0,\Tini+\Tf-1]}$ and $(u^\rmd_{[0,T-1]}, \hat{y}^\rmd_{[0,T-1]}) \in \hat{\calB}_{[0,\Tini+\Tf-1]}$.
We arrange them into Hankel matrices
\begin{equation}\label{eq:hankel-data}
\resizebox{0.9\columnwidth}{!}{$
H \coloneqq
\begin{bmatrix}
    \calH_{\Tini+\Tf}(u^\rmd_{[0,T-1]}) \\[3pt]
    \calH_{\Tini+\Tf}(y^\rmd_{[0,T-1]})
\end{bmatrix},
\hat{H} \coloneqq
\begin{bmatrix}
    \calH_{\Tini+\Tf}(u^\rmd_{[0,T-1]}) \\[3pt]
    \calH_{\Tini+\Tf}(\hat{y}^\rmd_{[0,T-1]})
\end{bmatrix}.
$}
\end{equation}
Given $(u_\ini, y_\ini) \in \calB_{[0, \Tini - 1]}$ and an input sequence $u \in \R^{m\Tf}$, 
we denote the predicted output sequence associated with these Hankel matrices by ${y}^\pred_{[0, \Tf-1]}(H)$ and $y^\pred_{[0, \Tf-1]}(\hat{H})$.
Note that $(u_\ini, y_\ini)$ need not be a $\Tini$-length trajectory of \eqref{eq:LTI-2}, since the subspace predictor $\calS$ is defined for arbitrary $(u_\ini, y_\ini) \in \R^{(m+p)\Tini}$. This corresponds to the practical setting in which the approximate subspace, $\hat{\calB}_{[0,\Tini+\Tf-1]}$, is used for prediction. 
Denote $b=(u_\ini, u, y_\ini)$.
Denote $M=[H_\up^\top H_\uf^\top H_\yp^\top]^\top$ and similar for $\hat{M}$, the prediction error can be first bounded as 
\begin{equation}\label{eq:first-error-bound}
    \begin{aligned}
    &\norm{y_{[0, \Tf-1]}^\pred(\hat{H})-y^\pred_{[0, \Tf-1]}(H)}_2 \\
    \leq &\left(\norm{\hat{H}_\yf}_2 \norm{\hat{M}^\dag - M^\dag}_2 + \norm{\hat{H}_\yf - H_\yf}_F
    \norm{M^\dag}_2\right) \norm{b}_2 \\
    \leq &\left(\norm{\hat{H}_\yf}_2 \norm{\hat{M}^\dag - M^\dag}_2 + \norm{\hat{H} - H}_F
    \norm{M^\dag}_2\right) \norm{b}_2
    \end{aligned},
\end{equation}
where the last inequality uses the fact that $\hat{H}_\yf$ and $H_\yf$ are submatrices of $\hat{H}$ and $H$, respectively. However, the bound in \eqref{eq:first-error-bound} is still expressed in terms of matrix norms. As such, it is not intrinsic to the underlying restricted behaviors, and the bound may vary under different matrix representations of the same subspace. Since our goal is to relate prediction error to a geometric discrepancy between restricted behaviors, 
it is natural to seek a  coordinate-free  bound formulated directly in terms of $d( \im \hat{H}, \im H)$, see Fig. \ref{fig:pred-error-illust}. 
\begin{figure}[t]
    \centering
    \includegraphics[width=0.5\columnwidth]{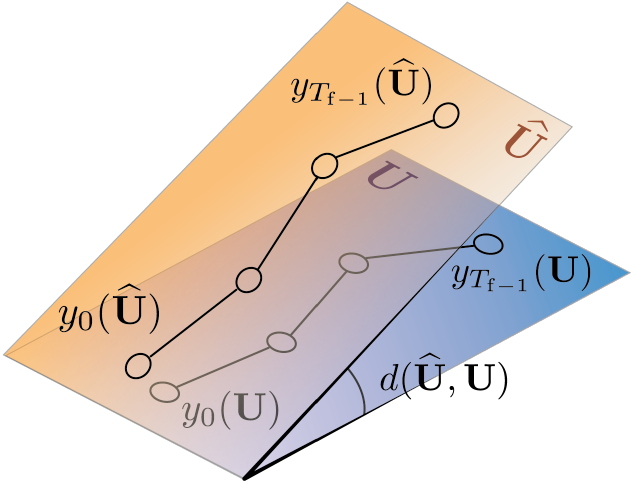}
    \caption{The two planes represent the nominal restricted behavior $\bmU$ and its approximation $\hat{\bmU}$, with discrepancy measured by $d(\hat{\bmU}, \bmU)$.}
    \label{fig:pred-error-illust}
\end{figure}
Because each restricted behavior is a linear subspace by Lemma \ref{thm:fund-lemma}, we use chordal distance $d$ to measure discrepancy between restricted behaviors, and will refer to chordal distance and behavioral distance interchangeably. 
This gives rise to the central question:
\begin{problem}\label{problem}
    How do perturbations measured in chordal distance between restricted behaviors affect the subspace prediction? More precisely, how can the prediction error
    \[
    \norm{y_{[0, \Tf-1]}^\pred(H) - y_{[0, \Tf-1]}^\pred(\hat{H})}_2
    \]
    be bounded in terms of $d(\im H, \im \hat{H})$?
\end{problem}

\noindent
We address this problem in Section \ref{sec:error-bound}. 
We first examine how observability properties of \eqref{eq:LTI-1} affect the prediction error bound given by \eqref{eq:error-bound}, then we show that the subspace predictor $\calS$ is  coordinate-free  ( Proposition \ref{thm:sp-rotinvar}).
Finally, we relate the Frobenius norm to behavioral distance through a suitable basis representation (Lemma \ref{lemma:fro-chordal}), which leads to a novel prediction error bound (Theorem \ref{thm:error-bound}).
\section{Prediction Error Bound}\label{sec:error-bound}
We introduce a state-space formulation of the restricted behavior. For $k \in \Zp$, define the extended observability matrix and block Toeplitz matrix as 
\[
\mathcal{O}_k =
\begin{bmatrix}
C \\ CA \\ \vdots \\ CA^{k-1}
\end{bmatrix},
\mathcal{T}_k
= \begin{bmatrix}
D         & 0         & \cdots & 0 \\
CB        & D         & \cdots & 0 \\
\vdots    & \vdots    & \ddots & \vdots \\
CA^{k-2}B & CA^{k-3}B & \cdots & D
\end{bmatrix}.
\]
Let $(u_{[0,L-1]},y_{[0,L-1]}) \in \calB_{[0,L-1]}$ with corresponding state sequence $x_{[0,L]}$.
Denote \[
\Phi_L \coloneqq \begin{bmatrix}
0 & I_{mL} \\
\mathcal{O}_L & \mathcal{T}_L
\end{bmatrix}, 
\]
then
\begin{equation}\label{eq:state-space-behavior}
\begin{bmatrix}
\mathcal{H}_L\!\left(u_{[0,T-1]}\right) \\[2pt]
\mathcal{H}_L\!\left(y_{[0,T-1]}\right)
\end{bmatrix}
=
\Phi_L
\begin{bmatrix}
\mathcal{H}_1\!\left(x_{[0,T-L]}\right) \\[2pt]
\mathcal{H}_L\!\left(u_{[0,T-1]}\right)
\end{bmatrix}.
\end{equation}
It is shown in \cite[Corollary 19]{markovsky2022identifiability} that $\im \Phi_L = \calB_{[0,L-1]}$.
We refer to $\Phi_L$ as the \textit{trajectory generation matrix} of \eqref{eq:LTI-1}. 
The main theorem below assumes a lower bound on $\sigma_{\min}(\Phi_L)$, a quantity tied to the observability property of \eqref{eq:LTI-1}, which is analyzed in Section \ref{sec:sigma_min-M}.
We now state the main theorem.
\begin{theorem}\label{thm:error-bound}
Consider system \eqref{eq:LTI-1}. Let $\Tini, \Tf \in \Z_{>0}$ and let $\Tini \geq n$. Let $\Phi_{\Tini}$ and $\Phi_{\Tini+\Tf}$ be trajectory generation matrices of \eqref{eq:LTI-1}.
Suppose there exists $\beta>0$ such that $\sigma_{\min}(\Phi_{\Tini}) \geq \beta$ and let $\sigma_{\max}(\Phi_{\Tini + \Tf}) = \alpha > 0$.
Let $(u_\ini, y_\ini) \in \calB_{[0,\Tini-1]}$, $u \in \R^{m\Tf}$.
Let $\bmU = \calB_{[0,\Tini+\Tf-1]}$ and $\hat{\bmU} \in \Gr((m+p)(\Tini+\Tf), \dim \bmU).$
Denote $\ga = \min(1, \beta)/\alpha$.
If $d({\bmU}, \hat{\bmU}) \leq \ga / 2\sqrt{2}$, 
then 
\begin{equation}\label{eq:error-bound}
    \begin{aligned}
    &\norm{y_{[0,\Tf-1]}^\pred(\hat{\bmU}) - y_{[0,\Tf-1]}^\pred(\bmU)}_2 \leq \\
    &\left(
    \frac{2(1+\sqrt{5})}{\ga^2} +
    \frac{1}{\ga}
    \right)
    \sqrt{2}d(\hat{\bmU}, \bmU)\,
    \|(u_\ini, u, y_\ini)\|_2 .
    \end{aligned}
\end{equation}
\end{theorem}
\noindent
Theorem \ref{thm:error-bound} answers Problem \ref{problem} by quantifying how perturbations of restricted behavior, measured in behavioral distance, affect the prediction error. 
The perturbation in \eqref{eq:error-bound} is quantified by the chordal
distance between $\hat{\bmU}$ and $\bmU$, which is
independent of the data-matrix representations.  However, the
resulting Lipschitz constant is expressed through $\Phi_{T_{\ini}}$ and
$\Phi_{T_{\ini}+T_f}$, and therefore depends on the chosen state-space
realization and input-output coordinates. Hence, the bound uses a coordinate-free subspace distance, but its explicit constant is realization-dependent. Theorem~\ref{thm:one-step-pred-error-known} provides a more practical, data-dependent counterpart, where these constants are replaced by quantities computed from data matrices.  

The equal-dimensional assumption that $\dim \hat{\bmU} = \dim \bmU$ in Theorem 1 ensures that the chordal distance is well-defined. This setting is natural for subspace tracking methods with a fixed target dimension \cite{sasfi2026great}. 
 It also covers truncated \gls{SVD}-type preprocessing \cite{kaviani2025uncertainty} when the noisy data matrix is truncated to the prescribed rank. In contrast, if a rank-selection strategy returns a dimension different from the true dimension, then the present chordal-distance bound does not directly apply. 
 Extensions to different-dimensional behavioral subspaces \cite{padoan2025distances} can be handled using distances on the disjoint union of Grassmannians, such as the Grassmannian metric in \cite[Proposition 15]{ye2016schubert}; we leave this direction for future work.
The proof of Theorem \ref{thm:error-bound}, presented in Section \ref{sec:proof}, relies on three auxiliary results, developed in the following sections.
\vspace{-10pt}
\subsection{ Coordinate-free  subspace predictor}\label{subsec:rot-invar-subspace-pred}
We establish that the subspace predictor depends only on the restricted behavior, rather than on the particular matrix representation. 
\begin{proposition}\label{thm:sp-rotinvar}
    Let $\Tini, \Tf \in \Z_{>0}$. Let $X, Y \in \R^{(m+p)(\Tini+\Tf) \times r}, r \in \Z_{>0}$. Suppose $\im X = \im Y$. Assume $M = [X_\up^\top X_\uf^\top X_\yp^\top]^\top$ has full column rank. 
     Then for any $(u_\ini, u, y_\ini) \in \R^{m(\Tini+\Tf)+p\Tini}$,
    \[
    \calS(X, u_\ini, u, y_\ini)=
    \calS(Y, u_\ini, u, y_\ini).
    \]
\end{proposition}

\begin{proof}
    Since $M$ has full column rank and $M$ is a submatrix of $X$, we have $\rank X = \rank Y = r$.
    Consider the QR decompositions \cite[Theorem 5.2.3]{golub2013matrix} $X=UR_X, Y=VR_Y$, where $U,V \in \R^{(m+p)(\Tini+\Tf) \times r}$ are orthonormal and $R_X, R_Y \in \R^{r \times r}$ are invertible. Since $\im X = \im Y$, we have $\im U = \im V$, hence $V=UQ$ for some $Q \in \rmO(r)$. 
    Denote $M_U = [U_\up^\top U_\uf^\top U_\yp^\top]^\top$, 
    then $M = M_U R_X$, which implies $M_U$ also has full column rank. Then $(M_U R_X)^\dag = R_X^{-1} M_U^\dag$. Therefore,
    $
    y_{[0,\Tf-1]}^\pred(X) 
    = (U_\yf R_X)(M_U R_X)^\dag b 
    = U_\yf R_X (R_X^{-1} M_U^\dag) b 
    = U_\yf M_U^\dag b.
    $
    Similarly, we have
    $y_{[0, \Tf-1]}^\pred(Y)    
    = U_\yf QR_Y (M_U QR_Y)^\dag b 
    = U_\yf QR_Y R_Y^{-1}(M_U Q)^\dag b
    = U_\yf QQ^\top M_U^\dag b = U_\yf M_U^\dag b. 
    $
    Thus $y_{[0, \Tf-1]}^\pred(X) = y_{[0, \Tf-1]}^\pred(Y)$.
\end{proof}
\vspace{2pt}
\noindent
The full-column-rank assumption is discussed quantitatively in Lemma \ref{lemma:sigma-min-M}, where it is related to $\Tini$ and the observability properties of the restricted behavior.
 Proposition \ref{thm:sp-rotinvar}  implies that $\calS$ can be viewed as a well-defined mapping on the Grassmannian, i.e., 
    \begin{align*}
    \calS: \Gr((m+p)(\Tini+\Tf), r) &\times \R^{m(\Tini+\Tf)+p\Tini} \to \R^{p\Tf}, \\
    \calS(\bmU, u_\ini, u, y_\ini) &=  U_\yf 
    \begin{bmatrix}
    U_\up  \\
    U_\uf  \\
    U_\yp 
    \end{bmatrix}^\dag 
    \begin{bmatrix}
        u_\ini \\ u \\ y_\ini
    \end{bmatrix},
    \end{align*}
    where $U$ is an orthonormal basis of $\bmU$.
When the vector $(u_\ini, u, y_\ini)$ is clear from context, the notations 
\[
\calS(\bmU), \quad 
\calS(U), \quad y_{[0, \Tf-1]}^\pred(U)
\]
will be used interchangeably. 
Now we can replace $H$ and $\hat{H}$ in \eqref{eq:first-error-bound} with their orthonormal bases $U$ and $\hat{U}$, respectively. Then \eqref{eq:first-error-bound} reveals two sources of prediction error: one through the sensitivity of $M^\dag$, governed by the singular values of $M$, and the other is related to the matrix norm $\|\hat{U} - U\|_F$ which depends on the basis representations of the restricted behaviors $\im \hat{H}$ and $\im H$. 
We next study the quantitative observability of \eqref{eq:LTI-1}, which is related to the singular values of $M$. 
\subsection{On the quantitative observability condition}
\label{sec:sigma_min-M}
We now turn to the assumption on $\sigma_{\min}(\Phi_{\Tini})$ in Theorem \ref{thm:error-bound}. The following proposition shows that this assumption is equivalent to observability of system \eqref{eq:LTI-1}.
\begin{proposition}
Let $\Tini \geq n$. Then there exists $\beta>0$ such that $\sigma_{\min}(\Phi_{\Tini})\geq \beta$ if and only if
$\rank(\mathcal O_{\Tini})=n.$
\end{proposition}
\begin{proof}
Positivity of $\sigma_{\min}(\Phi_{\Tini})$ implies that $\Phi_{\Tini}$ has full column rank $n+m\Tini$. The upper right block $I_{m\Tini}$ has rank $m\Tini$, so the remaining $n$ columns must be linearly independent, which is equivalent to $\rank(\mathcal O_{\Tini})=n$. Conversely, if $\rank (\calO_{\Tini}) = n$, then the structure of $\Phi_{\Tini}$ implies that $\Phi_{\Tini}$ has full column rank, and hence $\sigma_{\min}(\Phi_{\Tini}) > 0$.
\end{proof}
\noindent
This viewpoint also aligns with the quantitative observability condition \cite[Eq. (20)]{coulson2022quantitative} introduced in the robust fundamental lemma, which assumes a lower bound on $\sigma_{\min}(\Phi_{\Tini+\Tf})$. The parameter $\beta$ quantifies the \emph{degree} of observability.
\begin{lemma}\label{lemma:sigma-min-M}
Consider system \eqref{eq:LTI-1}. 
Let $T_{\mathrm{ini}}\ge n$, $T_f \in \Z_{>0}$. Let $V \in \R^{(m+p)(\Tini+\Tf) \times r}$ be an orthonormal basis of $\im \Phi_{\Tini+\Tf}$. 
Denote $M = [V_\up^\top~V_\uf^\top~V_\yp^\top]^\top$.
Assume that $\sigma_{\min}(\Phi_{\Tini}) \geq \beta$ for some $\beta>0$ and $\sigma_{\max}(\Phi_{\Tini+\Tf}) = \alpha > 0$.
Then $M$ satisfies
\[
\sigma_{\min}(M) \;\ge\; \min\{1,\beta\} / \alpha.
\]
The bound is also  coordinate-free  in the sense that $\sigma_{\min}(MP) \;\ge\; \min\{1,\beta\} / \alpha$ for all $P \in \rmO(r)$. 
\end{lemma}
\begin{proof}
Let $U$ be the orthonormal matrix from the QR factorization $\Phi_{\Tini+\Tf} = UR$, where $R$ is invertible.
Then there exists $Q \in \rmO(r)$ such that $U = VQ$. 
Let $\Pi$ denote the operator that selects the first $m\Tini+m\Tf+p\Tini$ rows of a matrix such that $M = \Pi V = \Pi UQ$.
Then 
$
\Pi \Phi_{\Tini+\Tf} = (\Pi VQ) R = M QR.
$
The matrix $\Pi \Phi_{\Tini+\Tf}$ admits the form
\[
\Pi \Phi_{\Tini+\Tf} =
\begin{bmatrix}
0 & I_{mT_{\mathrm{ini}}} & 0\\
0 & 0 & I_{mT_f}\\
\mathcal O_{T_{\mathrm{ini}}} & \mathcal T_{T_{\mathrm{ini}}} & 0
\end{bmatrix} \coloneqq \begin{bmatrix}
    F ~|~ G
\end{bmatrix}
\]
where $F=\begin{bmatrix}
    0 & I_{m\Tini} \\
    0 & 0 \\
    \calO_{\Tini} & \calT_{\Tini}
\end{bmatrix}$ and $G=\begin{bmatrix}
    0 \\ I_{m\Tf} \\ 0
\end{bmatrix}$. Since $F^\top G = 0$, the column spaces of $F$ and $G$ are orthogonal. Therefore for any vector $x=(x_1, x_2)$ where $x_1 \in \R^{m\Tini+p\Tini}, x_2 \in \R^{m\Tf}$, 
\begin{align*}
&\| \Pi \Phi_{\Tini+\Tf} x\|_2^2
= \| F x_1 \|_2^2 + \| G x_2 \|_2^2 \\
\geq &\sigma_{\min}(F)^2\|x_1\|_2^2 + \sigma_{\min}(G)^2\|x_2\|_2^2 \\
\geq &\min(\sigma_{\min}(F)^2, \sigma_{\min}(G)^2) (\|x_1\|_2^2 + \|x_2\|_2^2) \\
= &\min(\sigma_{\min}(F)^2, \sigma_{\min}(G)^2) \|x\|_2^2.
\end{align*}
Now $\sigma_{\min}(G) = 1$, and $\sigma_{\min}(F) = \sigma_{\min}(\Phi_{\Tini}) \geq \beta$.
Therefore, $\sigma_{\min}(\Pi \Phi_{\Tini+\Tf}) \geq \min \{\beta, 1\}.$
We have the inequality 
$
\|XYx\|_2 \geq \sigma_{\min}(X)\|Yx\|_2
\geq \sigma_{\min}(X)\sigma_{\min}(Y)\|x\|_2
$ for all $x \in \R^r$, where $X=\Pi \Phi_{\Tini+\Tf}, Y=(QR)^{-1}$.
Hence 
$\sigma_{\min}(XY) \geq \sigma_{\min}(X) \sigma_{\min}(Y)$.
Since $\Phi_{\Tini+\Tf}$ has full column rank and $R$ is invertible, 
\begin{align*}
&\sigma_{\min}(M)
= \sigma_{\min}(XY) \geq \sigma_{\min}(X)\sigma_{\min}((QR)^{-1}) \\
&= \sigma_{\min}(\Pi \Phi_{\Tini+\Tf}) / \sigma_{\max}(QR) \\
&= \sigma_{\min}(\Pi \Phi_{\Tini+\Tf}) / \sigma_{\max}(R)
\geq \min\{1,\beta\} / \alpha.
\end{align*}
Finally, since $\sigma(MP)=\sigma(M)$ for all $P \in \rmO(r)$, we have $\sigma_{\min}(MP) = \sigma_{\min}(M)$. 
\end{proof}
\noindent
Denote $\ga = \min(1, \beta) / \alpha$, then we see that the bound in Theorem \ref{thm:error-bound} improves as $\beta$ increases, and deteriorates as $\alpha$ increases. In this sense, a system with stronger quantitative observability leads to a tighter bound, while a larger gain $\alpha$ leads to greater sensitivity to behavioral perturbations.

\subsection{Basis alignment and chordal distance}
The term $\|\hat{U} - U\|_F$ in \eqref{eq:first-error-bound} depends on the particular choice of basis of $\hat{\bmU}$ and $\bmU$, whereas the prediction error $\norm{y_{[0,\Tf-1]}^\pred(\hat{U}) - y_{[0,\Tf-1]}^\pred (U)}_2$ remains invariant by  Proposition \ref{thm:sp-rotinvar} . 
It is therefore natural to choose orthonormal bases that minimize $\|\hat{U} - U\|_F$. The following lemma is classical and closely related to the orthogonal Procrustes problem \cite{schonemann1966generalized}. We include it here for completeness, since the proof highlights the geometric role of principal angles and aligns with the subspace-based viewpoint adopted in this work.
\begin{lemma}\label{lemma:fro-chordal}
    Let $\bmU, \hat{\bmU} \in \Gr((m+p)L, r)$ with orthonormal bases $U$ and $\hat{U}$, respectively. Let $\theta_1, \dots, \theta_r$ be the principal angles between $\bmU$ and $\hat{\bmU}$.  
    Then 
    \(
    \min_{R \in \rmO(r)} \norm{U - \hat{U}R}_F^2 = 
    2r - 2\sum_{i=1}^r \cos \theta_i.
    \)
    The minimum is attained at $R^* = QP^\top$, where 
    $U^\top \hat{U} = P \cos \Theta Q^\top$ is the compact \gls{SVD} with $\cos \Theta \coloneqq \diag(\cos \theta_1, \dots, \cos \theta_r)$. In particular, 
    \(
    \norm{U - \hat{U}R^*}_F \leq \sqrt{2}d(\bmU, \hat{\bmU}).
    \)
\end{lemma}
\begin{proof}
Let $R \in \rmO(r)$, we have $\|U - \hat{U}R\|_F^2 
= 2r - 2\operatorname{Tr} (U^\top \hat{U} R)$.
Denoting $S=Q^\top RP$, we have
\(\Tr(U^\top \hat{U} R)
= \sum_{i=1}^r S_{ii} \cos \theta_i. 
\)
Since $S = Q^\top RP \in \rmO(r)$, we have $|S_{ii}| \leq 1$ for $i=1,\dots, r$. Therefore, $\operatorname{Tr}(U^\top \hat{U} R)$ is maximized when $R = QP^\top$, in which case $S=I_r$. Then 
\[
\resizebox{\columnwidth}{!}{$\displaystyle 
    \norm{U - \hat{U}QP^\top}_F^2
    = 2r - 2\sum_{i=1}^r \cos \theta_i
    \leq 2\sum_{i=1}^r \sin^2 \theta_i \\
    = 2d(\hat{\bmU}, \bmU)^2,
    $}
\]
where the last inequality follows from $
\sin^2 \theta_i - (1 - \cos \theta_i) = \cos \theta_i (1 - \cos \theta_i) \geq 0$.
\end{proof}
\subsection{Proof of Theorem \ref{thm:error-bound}}\label{sec:proof}
By Lemma \ref{lemma:fro-chordal},
once an orthonormal basis $U$ of $\bmU$ is fixed, we may choose an orthonormal basis $\hat{U}$ of $\hat{\bmU}$ so that $\|\hat{U}-U\|_F$ is minimized. In the following, $\hat{U}$ denotes such a minimizing basis constructed in Lemma \ref{lemma:fro-chordal}.    Denote $\kappa = d(\hat{\bmU}, \bmU)$, we have
    $\|\hat{U}_\yf - U_\yf\|_F \leq \|\hat{U} - U\|_F \leq \sqrt{2}\kappa$. The difference between $\hat{M}^\dag$ and $M^\dag$ can be bounded using results on perturbation of pseudo-inverse \cite[Theorem 3.3]{stewart1977perturbation}, which implies
    \begin{equation*}
    \resizebox{\columnwidth}{!}{$\displaystyle
    \norm{\hat{M}^\dag - M^\dag}_2  
    \leq \frac{1+\sqrt{5}}{2} 
    \max\left\{\norm{\hat{M}^\dag}_2^2, 
          \norm{{M}^\dag}_2^2 \right\} \norm{\hat{M} - M}_2.
    $}
    \end{equation*}
    By Weyl's inequality \cite[Eq. (3.3.19)]{roger1994topics} for singular values and Lemma \ref{lemma:fro-chordal},
    \[
    \resizebox{\columnwidth}{!}{$\displaystyle 
    |\sigma_{\min}(\hat{M}) - \sigma_{\min}(M)| \leq
    \norm{\hat{M} - M}_2 
    \leq \norm{\hat{M} - M}_F \leq \sqrt{2}\kappa,    
    $}
    \]
    hence $\sigma_{\min}(\hat{M}) \geq \sigma_{\min}({M}) - \sqrt{2}\kappa \geq \ga - \sqrt{2}\kappa$ 
    by Lemma \ref{lemma:sigma-min-M}.
    Since $\kappa \leq \ga / {2\sqrt{2}}$, 
    then $\sigma_{\min}(\hat{M}) \geq \ga / {2}$, 
    so $\|\hat{M}^\dag\|_2 = \sigma_{\min}(\hat{M})^{-1} \leq {2} / \ga$. We then have 
    \[
    \max\left\{\norm{\hat{M}^\dag}_2^2, 
          \norm{{M}^\dag}_2^2 \right\} 
    \leq \max\left\{
    \frac{4}{\ga^2}, 
    \frac{1}{\ga^2} 
    \right\} \\
    = \frac{4}{\ga^2},     
    \]
Plugging the above inequalities into \eqref{eq:first-error-bound}, using Lemma \ref{lemma:fro-chordal}, and using the fact that $\|\hat{U}_\yf\|_2 \leq \|\hat{U}\|_2 = 1$, we have 
\begin{align*}
&\norm{\calS(\hat{\bmU}) - \calS(\bmU)}_2
= \norm{\calS(\hat{U}) - \calS(U)}_2 \notag\\
&\leq
\left(
\frac{2(1+\sqrt{5})\|\hat{U}_\yf\|_2}{\ga^2}
+
\frac{1}{\ga}
\right)
\norm{\hat{U}-U}_F\|b\|_2 \\
&\leq
\left(
\frac{2(1+\sqrt{5})}{\ga^2}
+
\frac{1}{\ga}
\right)
\sqrt{2}d(\hat{\bmU},\bmU)\|b\|_2 . 
\end{align*}

    The proof uses classical pseudoinverse perturbation tools, but Theorem 1 differs from classical subspace-angle perturbation bounds \cite{fierro1996perturbation}: it bounds the prediction error of the block-structured subspace predictor under chordal-distance perturbations, with constants containing system-theoretic quantities. Thus, the theorem is a Lipschitz-continuity result for the subspace predictor.

\subsection{ Data-dependent prediction error bound }

In practice, the prediction error bound \eqref{eq:error-bound} is not directly computable because the true behavior is unknown. In this section, we derive a data-dependent error bound, with the only unknown term being the chordal distance between the perturbed and the true subspaces. 
In some settings, however, this chordal distance can itself be bounded. For noisy data processed by truncated \gls{SVD}, \cite[Theorem 9]{alsalti2024robust} provides a computable bound for the chordal distance between the estimated and true subspaces. For recursive subspace identification, subspace tracking guarantees such as \cite[Theorem 1]{sasfi2026great} provide an online bound on the chordal distance of the estimated and true subspaces. 
Thus, Theorem~\ref{thm:one-step-pred-error-known} can be composed with these results to produce computable certificates.

\begin{theorem}\label{thm:one-step-pred-error-known}
Consider system \eqref{eq:LTI-1}. Let $\Tini, \Tf \in \Z_{>0}$ and let $\Tini \geq n$. 
Let $(u_\ini, y_\ini) \in \calB_{[0,\Tini-1]}$, $u \in \R^{m\Tf}$.
Let 
$\calB_{[0,\Tini+\Tf-1]} = \bmU$ and $\hat{\bmU} \in \Gr((m+p)(\Tini+\Tf), \dim \bmU)$. 
Let $\hat{U}$ be an orthonormal basis of $\hat{\bmU}$. Let 
$\hat{M} = [\hat{U}_{\up}^\top ~ \hat{U}_\uf^\top ~ \hat{U}_\yp^\top]^\top$.
Let $\hat{U}^{(1)}_\yf \coloneqq [I_p ~ 0 \cdots 0]\hat{U}_\yf$. 
If $d(\hat{\bmU}, \bmU) \leq \sigma_{\min}(\hat{M})/(2\sqrt{2})$ and $\sigma_{\min}(\hat{M})>0$, we have 
\begin{equation}\label{eq:error-bound-one-step-known}
    \begin{aligned}
    &\norm{y_0^\pred(\hat{\bmU}) - y_0^\pred(\bmU)}_2 \leq \\
    &\left( 
    \frac{2(1+\sqrt{5})\|\hat{U}^{(1)}_\yf\|_2}{\sigma_{\min}(\hat{M})^2} + 
    \frac{1}{\sigma_{\min}(\hat{M})}
    \right)
    \sqrt{2}d(\hat{\bmU}, \bmU)
    \|b\|_2.
    \end{aligned}
    \end{equation}
\end{theorem}

\begin{proof}
Let $U$ be an orthonormal basis of $\bmU$ and let $U_\yf^{(1)}$ be defined similarly as $\hat{U}_\yf^{(1)}$.
    Similar to $\eqref{eq:first-error-bound}$, we also have 
    \begin{align*}
    &\norm{y_0^\pred(\hat{\bmU}) - y_0^\pred(\bmU)}_2 \leq \\ 
    &\left(
    \norm{U^{(1)}_{\yf}}_2 \norm{M^\dag - \hat{M}^\dag}_2
    + \norm{\hat{U}^{(1)}_{\yf} - U^{(1)}_{\yf}}_F 
    \norm{\hat{M}^\dag}_2
    \right) \|b\|_2.
    \end{align*}
    Let $\kappa = d(\hat{\bmU}, \bmU)$. 
    By Weyl's inequality used in Section \ref{sec:proof},
    we have
\(
\sigma_{\min}(M)\ge \sigma_{\min}(\widehat M)-\sqrt{2}\kappa
\ge {\sigma_{\min}(\widehat M)} / {2},
\)
so that
$
\|M^\dagger\|_2=\sigma_{\min}(M)^{-1}
\le {2} / {\sigma_{\min}(\widehat M)}.
$
Therefore,
$
\max\{\|M^\dagger\|_2^2,\|\hat{M}^\dagger\|_2^2\}
\le {4} / \sigma_{\min}  (\hat{M})^2.
$ The rest of the proof proceeds similarly as in Section \ref{sec:proof}, which yields \eqref{eq:error-bound-one-step-known}.
\end{proof} 
\section{Numerical Experiment}\label{sec:numerical}

We evaluate the data-dependent bound \eqref{eq:error-bound-one-step-known} in the Monte Carlo setting of \cite{kaviani2025uncertainty}. 
We generate random stable \gls{LTI} systems of order $n$, with $n\in\{1,2\}$, $\Tini,\Tf\in\{1,2,3\}$, and $p,m$ chosen between $1$ and $n$. 
For each random system, we collect $100$ time steps of offline input-output data using inputs uniformly distributed in $(-1,1)$. 
The offline output data are then corrupted pointwise by additive noise $e(t)$ satisfying $\|e(t)\|_2\leq N$, where $N$ is the prescribed output-noise bound and $N \in [10^{-8}, 10^{-3}].$  
Further implementation details follow \cite{kaviani2025uncertainty}. 
Code reproducing the experiment is available at \url{github.com/DianJin-Frederick/subspace_prediction_under_behavioral_perturbation}.

Our first goal is to compare \eqref{eq:error-bound-one-step-known} with the prediction-error bound in \cite[Theorem 2]{kaviani2025uncertainty}. 
Since the chordal distance $d(\hat{\bmU},\bmU)$ may be unknown in practice, our second goal is to evaluate a fully computable version of \eqref{eq:error-bound-one-step-known} by replacing $d(\hat{\bmU}, \bmU)$ with the computable chordal distance bound
\begin{equation}\label{eq:alsalti-bound}
    d(\hat{\bmU}, \bmU) \leq \frac{C_{\theta}}{\delta_1 \delta_2} N,
\end{equation} 
where $C_\theta, \delta_1, \delta_2$ can be computed from data and are defined in \cite[Theorem 9]{alsalti2024robust}. 
Substituting \eqref{eq:alsalti-bound} into \eqref{eq:error-bound-one-step-known} removes the only non-computable term $d(\hat{\bmU},\bmU)$ and yields a prediction-error certificate computable from noisy data and the noise bound $N$.
We use the relative gap defined in \cite{kaviani2025uncertainty}, which measures the percentage difference between the right-hand side and the left-hand side of the corresponding prediction-error bound, normalized by the output magnitude.
Denote 
$\mathrm{L} = \norm{y_0^\pred(\hat{\bmU}) - y_0^\pred(\bmU)}_2$ and $\mathrm{R} = \left( 
    \frac{2(1+\sqrt{5})\|\hat{U}^{(1)}_\yf\|_2}{\sigma_{\min}(\hat{M})^2} + 
    \frac{1}{\sigma_{\min}(\hat{M})}
    \right)
    \sqrt{2}d(\hat{\bmU}, \bmU)
    \|b\|_2$, 
then 
\[
\text{relative gap}
\coloneqq
\frac{\mathrm{R}-\mathrm{L}}
{\|y_0^\pred(\bmU)\|_2}.
\]

\begin{figure}[h]
    \centering
    \includegraphics[width=\columnwidth]{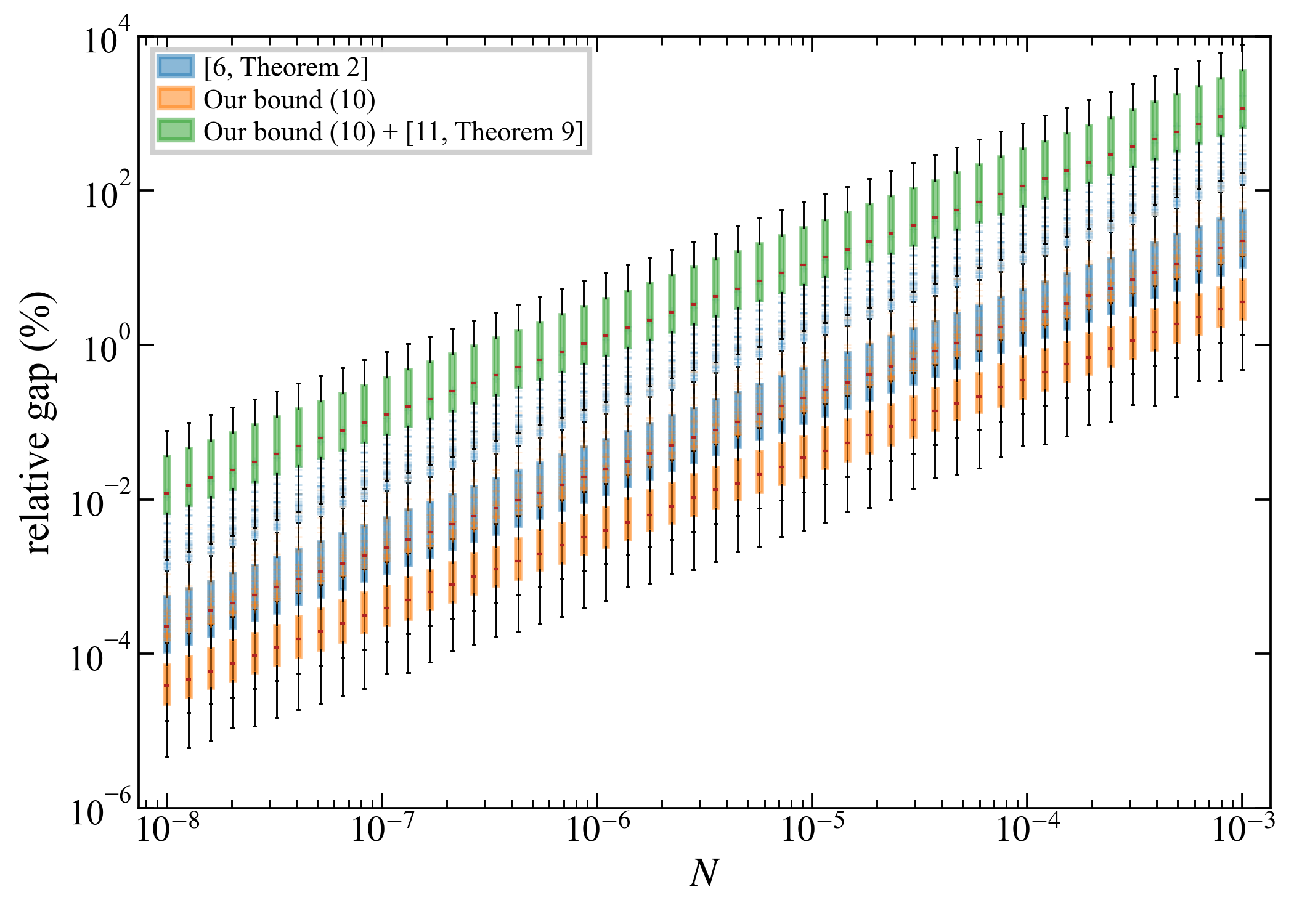}
    \caption{Relative gaps of prediction-error certificates versus the output-noise bound $N$. Boxes show the median and interquartile range; whiskers and crosses indicate nonoutlier ranges and outliers, respectively. 
    For readability, extreme outliers of the fully computable bound (green) are not shown.
    }
    \label{fig:box-plot}
\end{figure}
\vspace{10pt}
\noindent
\textbf{Results.} Fig. \ref{fig:box-plot} shows that \eqref{eq:error-bound-one-step-known}, when evaluated with the true chordal distance, has a smaller median relative gap than \cite[Theorem 2]{kaviani2025uncertainty} across all tested noise bounds. 
Thus, when the behavioral distance is available or accurately estimated, our chordal-distance-based bound gives a tighter prediction-error certificate than the data-matrix perturbation bound of \cite{kaviani2025uncertainty}. 
When the true chordal distance is replaced by the computable bound \eqref{eq:alsalti-bound}, the resulting certificate remains fully computable but becomes more conservative. 
This is expected, since the error is bounded in two stages: first from noisy data to a chordal-distance bound, and then from chordal distance to prediction error. 
The omitted green outliers correspond to cases where the chordal-distance bound \eqref{eq:alsalti-bound} is very conservative \cite[Theorem 14]{alsalti2024robust}. 
Thus, the experiment illustrates both the tightness of \eqref{eq:error-bound-one-step-known} over the baseline and the additional conservatism introduced when the bound is made fully computable.

\section{Conclusion}
This work shows two fundamental properties of the subspace predictor. 
First, we showed that it is  coordinate-free, in the sense that it depends only on the restricted behavior of a system. 
Second, we derived an explicit prediction error bound  directly in terms of the chordal distance between two restricted behaviors. 
This provides a quantitative link between geometric uncertainty and prediction accuracy.
Future work includes combining the results with computable subspace-tracking bounds and extending the analysis from prediction to closed-loop data-driven control. 
We hope that the  coordinate-free  bound developed will serve as a useful building block for 
more data-driven control methods.

\section*{Acknowledgement}
The authors used ChatGPT to assist with language refinement and editing in parts of this manuscript. All generated content was subsequently reviewed and edited by the authors, who take full responsibility for the final manuscript.